\documentclass{amsart}
\usepackage{amsmath,amssymb,graphicx,amsthm,bm,hyperref,mathrsfs,tikz,cite}
\usepackage{graphicx}
\usepackage{hyperref}
\usepackage[foot]{amsaddr}
\hypersetup{colorlinks=true,citecolor=blue,linkcolor=cyan}
\newtheorem{thm}{Theorem}[section]
\newtheorem{cor}[thm]{Corollary}
\newtheorem{lem}[thm]{Lemma}
\newtheorem{conj}[thm]{Conjecture}
\newtheorem{prop}[thm]{Proposition}
\theoremstyle{definition}
\newtheorem{defn}[thm]{Definition}
\theoremstyle{remark}
\newtheorem{rem}[thm]{Remark}
\numberwithin{equation}{section}

\newcommand{\Ff}{\mathbb{F}}
\newcommand{\Q}{\mathbb{Q}}
\newcommand{\R}{\mathbb{R}}
\newcommand{\C}{\mathbb{C}}
\newcommand{\Gal}{\textnormal{Gal}}
\newcommand{\Z}{\mathbb{Z}}
\newcommand{\F}{\mathcal{F}}

\newcommand{\ord}{\textnormal{ord}}

\newcommand{\supp}{\textnormal{supp}}

\begin{document}

\title[One Level Density for Cubic Galois Number Fields]{One Level Density for Cubic Galois Number Fields}%
\author{Patrick Meisner}%
\address{Tel Aviv University}%
\email{meisner@mail.tau.ac.il}%


\begin{abstract}

Katz and Sarnak predicted that the one level density of the zeros of a family of $L$-functions would fall into one of five categories. In this paper, we show that the one level density for $L$-functions attached to cubic Galois number fields falls into the category associated with unitary matrices.

\end{abstract}
\maketitle

\section{Introduction}\label{intro}

Given an $L$-function, the one-level density is the function
$$\mathscr{D}(L,f):=\sum_{\gamma} f\left(\frac{\gamma\log X}{2\pi}\right)$$
where $f$ is an even Schwartz test function and the sum runs over all non-trivial zeros of the $L$-function $\rho = 1/2+i\gamma$. The Generalized Riemann Hypothesis tells us that $\gamma$ will always be real. However, we do not suppose this.

\begin{rem}

The $\log$ factor in the definition of the one-level density is to ensure our zeros have mean spacing $1$.

\end{rem}

One may think of $f$ as a smooth approximation to the indicator function of an interval centered at $0$. Therefore the one-level density can be thought of as a measure of how many zeros are close to the real line, so called low-lying zeros.

For a suitably nice family $\mathcal{F}$ of $L$-functions and Schwartz function $f$ Katz and Sarnak \cite{KS} predicted that
$$\left\langle \mathscr{D}(L,f) \right\rangle_{\mathcal{F}} := \lim_{X\to\infty}\frac{1}{|\mathcal{F}(X)|} \sum_{L\in \mathcal{F}(X)}\mathscr{D}(L,f) = \int_{-\infty}^{\infty} f(t) W(G)(t) dt, $$
where the $\mathcal{F}(X)$ are finite increasing subsets of $\mathcal{F}$ and $W(G)(t)$ is the one-level density scaling of eigenvalues near 1 in a group of random matrices (indicated by $G$). This group, $G$, is called the symmetry type of the family $\mathcal{F}$.

Moreover, Katz and Sarnak predicted that $W(G)(t)$ would fall into one of these five categories
$$W(G)(t) = \begin{cases} 1 & G=U \\ 1 - \frac{\sin(2\pi t)}{2\pi t} & G=Sp \\ 1 + \frac{1}{2}\delta_0(t) & G=O \\ 1 + \frac{\sin(2\pi t)}{2\pi t} & G=SO(even) \\ 1 + \delta_0(t)-\frac{\sin(2\pi t)}{2\pi t} & G = SO(odd)  \end{cases}$$
where $\delta_0$ is the Dirac distribution and U, Sp, O, SO(even), SO(odd) are the groups of unitary, symplectic, orthogonal, even orthogonal and odd orthogonal matrices, respectively.

\subsection{Number Fields}

In this section, we will discuss known results for $L$-function attached to number fields.

%
%
%

For any number field, $K$, define $\zeta_K(s) = \sum_{\mathfrak{a}} N\mathfrak{a}^{-s}$. Denote $\zeta_{\mathbb{Q}}(s):=\zeta(s)$. Then the $L$-function associated to the field $K$ would be
$$L_K(s) = \frac{\zeta_K(s)}{\zeta(s)}.$$
Further, if we denote $D_K$ as the discriminant of $K$, then the one-level density will be
\begin{align} \label{numfieldOLD}
\mathscr{D}(K,f) = \sum_{\gamma} f\left(\frac{\gamma \log D_K}{2\pi}\right),
\end{align}
i.e. set $X=D_K$. Then Katz and Sarnak \cite{KS2} proved the following.
\begin{thm}\label{KSthm}
Let $\F(X)$ be the family of number fields of the form $\Q(\sqrt{8d})$ with $X\leq d \leq 2X$ and $d$ square-free. Then assuming GRH, if $\supp(\hat{f})\subset(-2,2)$, then
$$\lim_{X\to\infty} \frac{1}{|\F(X)|} \sum_{K\in \F(X)} \mathscr{D}(K,f) = \int_{\infty}^{\infty} f(t) W(Sp)(t) dt.$$
\end{thm}
Therefore, we see that the symmetry type for quadratic extensions is symplectic.

Further, in his thesis \cite{Y}, Yang considered the family of cubic non-Galois number fields.

\begin{thm}\label{Yangthm}

Let $N_3(X)$ denote the set of cubic fields of discriminant between $X$ and $2X$ and whose Galois closure is $S_3$. Then if $\supp(\hat{f})\subset (-1/50,1/50)$,
$$\lim_{X\to\infty} \frac{1}{|N_3(X)|} \sum_{K\in N_3(X)} \mathscr{D}(K,f) = \int_{\infty}^{\infty} f(t) W(Sp)(t) dt.$$

\end{thm}
Therefore, the symmetry type of cubic $S_3$-fields is symplectic as well.

\subsection{Function Fields}

Every finite extension of $\Ff_q(t)$ corresponds to a smooth projective curve $C$. We define the zeta-function of the curve as
$$Z_C(u) = \exp\left(\sum_{n=1}^{\infty} N_n(C) \frac{u^n}{n}\right)$$
where $N_n(C)$ is the number of $\Ff_{q^n}$-rational points on $C$. Since the GRH is known (proved by Weil in \cite{Weil}) we have
$$Z_C(u) =\frac{L_C(u)}{(1-u)(1-qu)}$$
where $L_C(u)$ is a polynomial that satisfies the function equations
$$L_C(u) = (qu^2)^g L_C\left(\frac{1}{qu}\right)$$
where $g$ is the genus of the curve $C$ and all its roots lie on the ``half-line" $|u|= q^{-1/2}$. Hence, we can find a unitary symplectic $2g\times2g$ matrix $\Theta_C$, called the Frobenius class of $C$, such that
$$L_C(u) = \det(I-u\sqrt{q} \Theta_C).$$
Then the zeros of $L_C(u)$ correspond to the eigenangles of $\Theta_C$.

Since the eigenangles of $\Theta_C$ are $2\pi$-periodic we need to modify the one-level density definition a bit. So, for an even Schwartz test function $f$, define
$$F(\theta) = \sum_{k\in \mathbb{Z}} f\left(N\left(\frac{\theta}{2\pi}-k\right)\right)$$
so that $F$ is $2\pi$-periodic and centered on an interval of size roughly $1/N$. Then for any $N\times N$ unitary matrix $U$ with eigenangles $\theta_1\dots,\theta_N$, define
$$Z_f(U) = \sum_{j=1}^N F(\theta_j).$$
Finally, we then get that the one-level density for $C$ will be
$$\mathscr{D}(L_C,f) = Z_f(\Theta_C).$$

The literature on the one-level density in the function field setting give slightly different predictions than in the number field setting. For a suitably nice family of curves $\mathcal{F}$ and even Schwartz function $f$, the literature predicts
$$\frac{1}{|\mathcal{F}(X)|} \sum_{C\in \mathcal{F}(X)} Z_f(\Theta_C) = \int_{G} Z_f(U) dU + o(1)$$
where $G$ is the symmetry type and $dU$ is the Haar measure.

Specifically, Rudnick \cite{R} proved the following.

\begin{thm}\label{Rudthm}

Let $q$ be odd and $\mathcal{F}_{2g+1}$ be the set of hyperelliptic curves with affine model $C: Y^2 = f(X)$ with $\deg(f)=2g+1$ (and thus the genus of $C$ is $g$). Then if $\supp(\hat{f})\subset(-2,2)$,
$$ \frac{1}{|\mathcal{F}_{2g+1}|} \sum_{C\in\mathcal{F}_{2g+1}} Z_f(\Theta_C) = \int_{USp(2g)} Z_f(U) dU + O\left(\frac{1}{g}\right)$$

\end{thm}
Hence, the symmetry type of hyperelliptic curves is $USp(2g)$. This is to be expected as all these curves correspond to quadratic extensions and Theorem \ref{KSthm} show that quadratic extensions in the number field setting have symmetry type Sp.

Now, Bucur, Costa, David, Guerreiro and Lowry-Duda \cite{BCDG+} prove the following.
\begin{thm}\label{BCDGthm}
Let $E_3(g)$ be the family of cubic non-Galois extension of $\Ff_q(X)$ with discriminant of degree $2g+4$. Then there exists a $\beta>0$ such that if $\supp(\hat{f})\subset(-\beta,\beta)$ then
$$ \frac{1}{|E_3(g)|} \sum_{C\in E_3(g)} Z_f(\Theta_C) = \int_{USp(2g)} Z_f(U) dU + O\left(\frac{1}{g}\right).$$
\end{thm}
This again, matches with what is know from the number field case in Theorem \ref{Yangthm} as a cubic non-Galois extension would have Galois closure $S_3$.

Finally, in the same paper Bucur, Costa, David, Guerreiro and Lowry-Duda extend Rudnick's result.
\begin{thm}\label{BCDGthm2}
Let $\ell$ be an odd prime, $q\equiv 1 \mod{\ell}$ and $\F_{g,\ell}$ be the moduli space of curves of $\ell$ covers of genus $g$. Then if $\supp(\hat{f})\subset(-\frac{1}{\ell-1},\frac{1}{\ell-1})$
$$ \frac{1}{|\F_{g,\ell}|} \sum_{C\in \F_{g,\ell}} Z_f(\Theta_C) = \int_{U(2g)} Z_f(U) dU + O\left(\frac{1}{g}\right).$$
\end{thm}
Here, we see a new symmetry type, that of $U(2g)$.

\subsection{Main Theorem}

The aim of this paper is to calculate the one-level density over cubic Galois number fields. Noticing the parallels in the function field setting and the number field setting we can use Theorem \ref{BCDGthm2} to predict that the symmetry type we should expect is $U$. Indeed that is what we find.

\begin{thm}\label{Mainthm}

Let $\F_3(X)$ be the family of cubic, Galois number fields of discriminant between $X$ and $2X$. Then if $f$ is an even Schwartz test function with $\supp(\hat{f})\subset(-1/14,1/14)$, we have
$$ \frac{1}{|\F_3(X)|} \sum_{K\in \F_3(X)}\mathscr{D}(K,f) = \int_{-\infty}^{\infty} f(t) W(U)(t)dt + O\left(\frac{1}{\log X}\right).$$
Moreover, if we assume GRH, then we can take $f$ with $\supp(\hat{f})\subset(-1/2,1/2)$.

\end{thm}

The proof of Theorem \ref{Mainthm} only relies on two facts: $(1)$ $3$ is a prime and $(2)$ that $\Z[\zeta_3]$ is a PID. Therefore, the same arguments could be extended to the family of $\Z/p\Z$ Galois number fields where $p$ is an odd prime such that $\Z[\zeta_p]$ is a PID. Unfortunately, these conditions are very limited as this is only true for primes less than $20$. However, with this and Theorem \ref{BCDGthm2} it is reasonable to conjecture the following.

\begin{conj}

Let $p$ be an odd prime and $\F_p(X)$ be the family of $\Z/p\Z$ Galois number fields of discriminant between $X$ and $2X$. Then there exists a $\beta>0$ (dependent only on $p$) such that for every even Schwartz test function $f$ such that $\supp(\hat{f})\subset(-\beta,\beta)$, we have
$$\lim_{X\to\infty} \frac{1}{|\F_p(X)|} \sum_{K\in \F_p(X)}\mathscr{D}(K,f) = \int_{\infty}^{\infty} f(t) W(U)(t)dt.$$

\end{conj}

\textbf{Acknowledgements:} I would like to thank Ze\'ev Rudnick for the useful discussions.

The research leading to these results has received funding from the European Research Council under the European Union's Seventh Framework Programme (FP7/2007-2013) / ERC grant agreement n$^{\text{o}}$ 320755.

\section{Classifying Cubic Galois Extensions}\label{classify}

In this section we will give a construction for all cubic Galois extensions of $\Q$.

\subsection{Class Field Theory}

We will begin by stating some main results of class field theory. For general reference we refer the reader to \cite{CF}.

Let $K$ be a global field. Denote $\mathcal{D}(K)$ the group of divisors of $K$. For any effective divisor $\mathfrak{m}\in\mathcal{D}(K)$, define
\begin{align*}
\mathcal{D}_{\mathfrak{m}}(K) & = \{D\in\mathcal{D}(K) : \supp(D)\cap\supp(\mathfrak{m})=\emptyset\} \\
\mathcal{P}_{\mathfrak{m}}(K) & = \{(a) : a\in K^*, a \equiv 1 \mod{\mathfrak{P}^{\ord_{\mathfrak{P}} (\mathfrak{m})}}  \mbox{ for all places $\mathfrak{P}$ of $K$} \} \\
\mathcal{C}\ell_{\mathfrak{m}}(K) & = \mathcal{D}_{\mathfrak{m}}(K)/\mathcal{P}_{\mathfrak{m}}(K).
\end{align*}
$\mathcal{P}_{\mathfrak{m}}(K)$ is the \textbf{ray} of $K$ modulo $\mathfrak{m}$ and $\mathcal{C}\ell_{\mathfrak{m}}(K)$ is the \textbf{ray class group} of $K$ modulo $\mathfrak{m}$.

\begin{thm}

There is a one-to-one correspondence between finite abelian Galois extensions $L$ of $K$ unramified outside of $\mathfrak{m}$ with Galois group $G$ and subgroups $H$ of $\mathcal{C}\ell_{\mathfrak{m}}(K)$ such that
$$G \cong \mathcal{C}\ell_{\mathfrak{m}}(K)/H $$

\end{thm}

If we set $K=\Q$, $\mathcal{D}(\Q) \cong \Q_{\geq0}$ and effective divisors correspond to positive integers. Hence, we will write an effective divisor of $\Q$ as $m$ instead of $\mathfrak{m}$ to illustrate that it is an integer. Further, we will denote $\supp(m)$ as the set of primes dividing $m$. Therefore, from the definitions, we get that
$$\mathcal{C}\ell_m(\Q) = \left(\Z/m\Z\right)^*.$$
Moreover, if we want to find subgroups of $\mathcal{C}\ell_m(\Q)$ such that
$$\mathcal{C}\ell_m(\Q)/H \cong \Z/3\Z$$
it suffices to look for subgroups of index $3$ of the three torsion subgroup of the ray class group:
$$\mathcal{C}\ell_m(K)[3] = \left(\Z/3\Z\right)^{\delta_{\mathfrak{m}}}\times\prod_{\substack{p|m \\ p\equiv 1 \mod{3}}} \Z/3\Z$$
where $\delta_m=1$ if $9|m$ and $0$ otherwise. Finally, since $\mathcal{C}\ell_m(K)[3]$ is a finite abelian group, subgroups of index $3$ are in one-to-one correspondence with subgroups isomorphic to $\Z/3\Z$.

Before we state the next result, we need a definition.

\begin{defn}

Call an integer \textbf{$3$-split} if all it's prime divisors are congruent to $0$ or $1 \mod{3}$.

\end{defn}

\begin{lem}\label{classifylem}

For any integer $m$, there is a two-to-one correspondence between cube-free $3$-split integers, $D$, such that $\supp(D)\subset\supp(m)$ and cubic Galois extensions of $\Q$ unramified outside of the primes dividing $m$.

\end{lem}

\begin{proof}

As was stated above there is a one-to-one correspondence between $\Z/3\Z$ subgroups of $\mathcal{C}\ell_m(\Q)[3]$ and cubic Galois extensions of $\Q$ unramified outisde of the primes dividing $m$. There is a one-to-two correspondence between such subgroups and non-zero elements of $$\mathcal{C}\ell_m(\Q)[3] = \left(\Z/3\Z\right)^{\delta_m}\times\prod_{\substack{p|m \\ p\equiv 1 \mod{3}}} \Z/3\Z.$$
Let $e_p$ be the coordinates of a element in $\mathcal{C}\ell_m(\Q)[3]$. Now we construct the cube-free $3$-split integer as
$$D := \prod_{\substack{p|m \\ p\equiv 0,1\mod{3}}}p^{e_p}.$$
This correspondence is one-to-two since there are two generators for each subgroup.

\end{proof}

\begin{cor}\label{classifycor}
Let $D_1$ and $D_2$ be two distinct cube-free $3$-split integers. Then they correspond to the same cubic Galois extension of $\Q$ if and only if there exists a $D\in\Q$ such that $D_2=D_1^2D^3$.

\end{cor}

\begin{proof}

Let
$$D_i = \prod p^{e_{p,i}}$$
be the prime factorization of $D_i$, $i=1,2$. Then by the proof of Lemma \ref{classifylem}, we see that $D_1$ and $D_2$ correspond to the same cubic extension of $\Q$ if and only if the vectors $(e_{p,1})$ and $(e_{p,2})$ generate the same subgroup in $\mathcal{C}\ell_m(\Q)[3]$ where $m$ is any positive integer such that $\supp(D_1)\cup\supp(D_2)\subset\supp(m)$. Since $D_1\not=D_2$, this is if and only if $e_{p,2} \equiv 2 e_{p,1} \mod{3}$ for all primes $p$. Setting
$$D = \prod p^{\frac{e_{p,2}-2e_{p,1}}{3}}$$
suffices.

\end{proof}

\subsection{Explicit Correspondence}

In this section, we will construct an explicit correspondence between cube-free $3$-split integers and cubic Galois extensions of $\Q$.

Let $\zeta_3$ be a primitive cubic root of unity and denote $K = \Q(\zeta_3)$. The following are well known facts about the cyclotomic field $K$.

\begin{lem}\label{knownlem}

\begin{enumerate}
\item The only ramified prime in $K$ is $3$ and a prime $p$ splits if $p\equiv 1 \mod{3}$ and is inert if $p\equiv 2 \mod{3}$.
\item $\mathcal{O}_K = \Z[\zeta_3]$ is a PID.
\item $\mathcal{O}_K^* = \{\pm1, \pm\zeta_3, \pm\zeta_3^2\}$
\item $K/Q$ is Galois with $\Gal(K/\Q) = \Z/2\Z$
\end{enumerate}

\end{lem}

Denote $\mathfrak{P}_3$ as the unique prime dividing $3$ in $\mathcal{O}_K$. Hence $3\mathcal{O}_K = \mathfrak{P}_3^2$. Moreover, denote $\sigma$ as the unique generator of $\Gal(K/\Q)$.

\begin{lem}\label{3-factorization}

Let $D$ be a $3$-split integer. Then there exists $D_1,D_2\in\mathcal{O}_K$ such that $D=\pm D_1D_2$, $\sigma(D_1)=D_2$ and $\gcd(D_1,D_2) = \mathfrak{P}_3^{v_3(D)}$.

\end{lem}

\begin{proof}

Since $D$ is $3$-split, we can write
$$D= 3^{e_3} \prod_{\substack{p|D \\ p\not=3}} p^{e_p}$$
where all the primes appearing in the product have the property that $p\equiv 1 \mod{3}$ and hence split in $K$. That is, we can write
$$p\mathcal{O}_K = \mathfrak{P}_1\mathfrak{P}_2$$
where $\mathfrak{P}_1^{\sigma} = \mathfrak{P}_2$.

Define
$$\mathscr{D}_i := \prod_{\substack{p|D \\ p\not=3}} \mathfrak{P}_i^{e_p}.$$
Since $\mathcal{O}_K=\Z[\zeta_3]$ is a PID, we can find $D'_i$ such that $\mathscr{D}_i=(D'_i)$. Moreover, since $\mathscr{D}_1^{\sigma} = \mathscr{D}_2$, we may assume $\sigma(D'_1)=D'_2$. Now, we notice that $3=(1-\zeta_3)(1-\bar{\zeta}_3)$. Define
$$D_1 = (1-\zeta_3)^{e_3}D'_1 \quad \quad \quad D_2 = (1-\bar{\zeta}_3)^{e_3}D'_2.$$
Then $\sigma(D_1)=D_2$ and $D\mathcal{O}_K =(D_1D_2)$. Therefore $D=uD_1D_2$ for some unit $u$ of $\mathcal{O}_K$. However, since both $D$ and $D_1D_2$ are fixed by $\sigma$, we see that $u$ is also fixed by $\sigma$ so $u=\pm1$.

Finally, we remark that $\gcd(D'_1,D'_2)=1$ and $\mathfrak{P}_3 = (1-\zeta_3)\mathcal{O}_K = (1-\bar{\zeta_3})\mathcal{O}_K$.

\end{proof}

\begin{defn}

For any $3$-split integer, we will call the factorization $D=\pm D_1D_2$ as in Lemma \ref{3-factorization} its \textbf{$3$-split factorization}.

\end{defn}

\begin{rem}\label{3splitrem}

The $3$-split factorization of an integer is not unique. It depends on choices of primes $\mathfrak{P}\in\mathcal{O}_K$ dividing $3$-split primes $p\in \Z$. As we will see the classification depends on the choice of factorization of the $3$-split primes in $\mathcal{O}_K$. However when we count such extensions this choice will not matter. Therefore, for every $3$-split prime, we will fix a prime in $\mathcal{O}_K$ dividing it and, consequently, fix a $3$-split factorization of all $3$-split integers.

\end{rem}

\begin{lem}\label{3splitfact}

For any $3$-split integer $D$ with $3$-split factorization $D=\pm D_1D_2$, the extension $K'_D:= \Q(\zeta_3,\sqrt[3]{D_1D_2^2})$ is a Galois extension of $\Q$ with Galois group $\Z/6\Z$.

\end{lem}

\begin{proof}

By Kummer theory we have that $K'_D$ is a Galois extension of $K$ with Galois group $\Z/3\Z$ (since $\mu_3 \subset K$). Let $\tau$ be a generator of $\Gal(K'_D/K)$ such that $\tau\left(\sqrt[3]{D_1D_2^2}\right) = \zeta_3\sqrt[3]{D_1D_2^2}$ and let $\sigma$ be the generator of $\Gal(K/\Q)$ as above.

We know that $\sigma(D_1D_2^2) = D_1^2D_2$ and so, up to a choice of cube root of $D_1^2D_2$, we get $\sigma\left(\sqrt[3]{D_1D_2^2}\right) = \sqrt[3]{D_1^2D_2}$. Therefore, $K'_D$ is a Galois extension of $\Q$.

Thus, $\sigma$ is an element of order $2$ and $\tau$ is an element of order $3$ in $\Gal(K'_D/\Q)$. Hence, $\sigma$ and $\tau$ generate $\Gal(K'_D/\Q)$ since $[K'_D:\Q]=6$. So it remains to show that $\sigma$ and $\tau$ commute.

Clearly $\sigma\tau(\zeta_3)=\tau\sigma(\zeta_3)$ since $\tau$ fixes $K$. Now,
$$\sigma\tau\left(\sqrt[3]{D_1D_2^2}\right) = \sigma\left(\zeta_3\sqrt[3]{D_1D_2^2}\right) = \bar{\zeta}_3\sqrt[3]{D_1^2D_2}$$
and
$$\tau\sigma\left(\sqrt[3]{D_1D_2^2}\right) = \tau\left(\sqrt[3]{D_1^2D_1}\right) = \tau\left(\frac{\sqrt[3]{D_1D_2^2}^2}{D_2}\right) = \frac{\zeta_3^2\sqrt[3]{D_1D_2^2}^2}{D_2} = \bar{\zeta_3}\sqrt[3]{D_1^2D_2}.$$

Therefore, $\sigma$ and $\tau$ commute and $\Gal(K'_D/\Q)=\Z/6\Z$ as claimed.

\end{proof}

Let $H=\{1,\sigma\}\subset \Gal(K'_D/\Q)$ and let $K_D = (K'_D)^H$ be the fixed field of $H$. Then
\begin{align}\label{KD}
K_D = \Q\left(\sqrt[3]{D_1D_2^2} + \sqrt[3]{D_1^2D_2}\right)
\end{align}
is Galois with $\Gal(K_D/\Q) = \Z/3\Z$.

\begin{lem}\label{explcorrlem}

Let $D_1,D_2$ be distinct $3$-split integers then $K_{D_1}=K_{D_2}$ if and only if there exists a $D\in\Q$ such that $D_2=D_1^2D^3$.

\end{lem}

\begin{proof}

Since $K'_{D_i} = K_{D_i}(\zeta_3)$ and $K_{D_i}=(K'_{D_i})^H$ we have $K_{D_1}=K_{D_2}$ if and only if $K'_{D_1}=K'_{D_2}$.

Let $D_1=\pm D_{1,1}D_{1,2}$, $D_2=\pm D_{2,1}D_{2,2}$ be the $3$-split factorization of $D_1$ and $D_2$. Then Kummer Theory applied to $K$ tells us that $K'_{D_1}=K'_{D_2}$ if and only if there exists $E\in K^*$ such that
\begin{align}\label{explproofeq}
D_{2,1}D_{2,2}^2 = D_{1,2}D_{1,1}^2E^3.
\end{align}
Let $p\not=3$ be a prime divisor of $D_2$ and let $\mathfrak{P}$ be a prime lying above it in $\mathcal{O}_K$. First suppose $\mathfrak{P}|D_{2,1}$. Then $\mathfrak{P}\nmid D_{2,2}$ and hence $v_{\mathfrak{P}}(D_{2,1}) = v_p(D_2)$. Moreover, from how we construct the $3$-split factorizations in Lemma \ref{3splitfact} we have that $\mathfrak{P} \nmid D_{1,2}$ and hence $v_{\mathfrak{P}}(D_{1,1}) = v_p(D_1)$.

Putting this together with \eqref{explproofeq} we get
$$v_P(D_2) = v_{\mathfrak{P}}(D_{2,1}D_{2,2}^2) = v_{\mathfrak{P}}(D_{1,2}D_{1,1}^2E^3) = 2v_p(D_1) + 3v_{\mathfrak{P}}(E).$$
Similarily, if $\mathfrak{p}|D_{2,2}$, we would get
$$2v_p(D_2) = 4v_p(D_1) + 3v_{\mathfrak{p}}(E).$$
In either case we get
$$v_p(D_1) \equiv 2v_p(D_2) \mod{3}$$
for all $p|D_2$, $p\not=3$.

Finally, if we let $\mathfrak{P}_3$ be the unique prime lying over $3$ in $K$, and consider just the powers of $\mathfrak{P}_3$ appearing in \eqref{explproofeq} then by the construction of the $3$-split factorization we get
$$3^{v_3(D_2)}(1-\zeta_3)^{v_3(D_2)} = 3^{v_3(D_1)}(1-\bar{\zeta_3})^{v_3(D_1)} E_3^3$$
where $E_3$ is the part of $E$ divisible by $\mathfrak{P}_3$. Using the fact the $1-\zeta_3 = 3/(1-\bar{\zeta_3})$ and rearranging we get
$$3^{2v_3(D_2)-v_3(D_1)} = (1-\bar{\zeta_3})^{v_3(D_1)+v_3(D_2)} E_3^3.$$
Hence, $E_3^3 = 3^n(1-\zeta_3)^{v_3(D_1)+v_3(D_2)}$. In particular, $v_3(D_1)+v_3(D_2) \equiv 0 \mod{3}$ and so $v_3(D_1)\equiv 2v_3(D_2) \mod{3}$ as required.

\end{proof}

\begin{prop}\label{classifyprop}

The two-to-one correspondence from cube-free $3$-split integers $D$ such that $\supp(D)\subset \supp(m)$ to cubic Galois extensions of $\Q$ unramified outside the primes dividing $m$, as in Lemma \ref{classifylem}, can be explicitly given by
$$D \mapsto K_D = \Q\left(\sqrt[3]{D_1D_2^2} + \sqrt[3]{D_1^2D_2}\right)$$

\end{prop}

\begin{proof}

We must first show that this map is well defined. That is, that $K_D$ is cubic, Galois and unramified outside of the primes dividing $m$. We have already shown that $K_D$ is in fact cubic and Galois. Since $[K:\Q]=2$ is coprime to $3=[K_D:\Q]=[K'_D:K]$, we see that a prime ramifies in $K_D$ if and only a prime lying above it in $\mathcal{O}_K$ ramifies in $K'_D$ if and only $p|D$. Therefore, the map is well defined. Finally, Lemmas \ref{classifylem}, \ref{explcorrlem} and Corollary \ref{classifycor} shows that this map is two-to-one and surjective.

\end{proof}

From now on $D$ will always denote a cube-free $3$-split integer.

\subsection{Discriminant}

Denote $\Delta_D$ as the discriminant of $K_D$. If we let $f_D$ be the conductor of $K_D$ then we have $\Delta_D = f_D^2$. Theorem 10 of \cite{H} states that $v_p(f)=1$ or $0$ if $p\not=3$ while $v_3(f) = 2$ or $0$. Thus we get that
\begin{align}\label{Disc1}
\Delta_D = 3^{4\delta_D} \prod_{p \mbox{ ramified in $K_D$}} p^2,
\end{align}
where $\delta_D$ is $1$ if $3$ is ramified in $K_D$ and $0$ otherwise. Therefore, it remains to determine which primes ramify in $K_D$.

As was mentioned in the proof of Proposition \ref{classifyprop} a prime $p$ ramifies in $K_D$ if and only $p|D$. Since $D$ is cube-free we can find $d_1$, $d_2$ square-free, coprime and coprime to $3$ such that $D=3^{v_3(D)}d_1d_2^2$. Then we have
\begin{align}\label{Disc2}
\Delta_D = (9^{\delta_D} d_1d_2)^2,
\end{align}
where $\delta_D$ is $1$ if $3|D$ and $0$ otherwise. (Note that this definition of $\delta_D$ agrees with the definition in \eqref{Disc1} as $3$ is ramified if and only if $3|D$.)

Finally, recall that $\F_3(X)$ is the set of cubic, Galois extensions of determinant between $X$ and $2X$. Then Theorem 1.2 of \cite{W} states that there exists a constant $c$, such that
\begin{align}\label{F3size}
|\F_3(X)| \sim cX^{1/2}.
\end{align}

\section{L-Functions and Explicit Formula}

Before we begin, we will fix some notation. We will denote $p$ as a prime in $\Q$, $\mathfrak{p}$ as a prime in $K_D$ and $\mathfrak{P}$ as a prime in $K=\Q(\zeta_3)$. Hence when we write an infinite product over primes, the set of primes that we run over will be indicated by which of the above three symbols we use. Moreover, we will denote $N\mathfrak{p}$ and $N\mathfrak{P}$ as the norms of $\mathfrak{p}$ and $\mathfrak{P}$ over $\Q$. Later, in Section \ref{appexplfor}, we will also use $\ell$ to denote a prime in $\Q$ and $\mathfrak{l}$ a prime dividing it in $K$ and $N\mathfrak{l}$ to denote the norm over $\Q$.

For any prime $p$ denote $e(p)$ and $f(p)$ as the ramification index and inertial degree of $p$ in $K$ and $e_D(p)$ and $f_D(p)$ as the ramification index and inertial degree of $p$ in $K_D$. Further, let $g(p)$ and $g_D(p)$ be the number or primes dividing $p$ in $K$ and $K_D$, respectively.

\subsection{L-Functions}

Let $\zeta(s),\zeta_K(s)$ and $\zeta_D(s)$ be the $\zeta$-functions of $\Q,K$ and $K_D$, respectively. That is,
\begin{align}\label{zeta}
\zeta(s) = \prod_p \left(1-\frac{1}{p^s}\right)^{-1}
\end{align}
\begin{align}\label{zetaK1}
\zeta_K(s) = \prod_{\mathfrak{P}} \left(1 - \frac{1}{N\mathfrak{P}^{s}}\right)^{-1}
\end{align}
\begin{align}\label{zetaD1}
\zeta_D(s) = \prod_{\mathfrak{p}} \left(1 - \frac{1}{N\mathfrak{p}^{s}}\right)^{-1}.
\end{align}
which all converge for $\Re(s)>1$.

Let
\begin{align}\label{L_K}
L_K(s) = \frac{\zeta_K(s)}{\zeta(s)}
\end{align}
\begin{align}\label{L_D}
L_D(s) = \frac{\zeta_D(s)}{\zeta(s)}
\end{align}
be the $L$-functions of $K$ and $K_D$, respectively.

Since both $K$ and $K_D$ are Galois, we can rewrite $\zeta_K$ and $\zeta_D$ as
\begin{align}\label{zetaK2}
\zeta_K(s) = \prod_p \left(1-\frac{1}{p^{f(p)s}}\right)^{-g(p)}
\end{align}
\begin{align}\label{zetaD2}
\zeta_D(s) = \prod_p \left(1-\frac{1}{p^{f_D(p)s}}\right)^{-g_D(p)}.
\end{align}

From Lemma \ref{knownlem} we have that
\begin{align}\label{efg1}
(e(p),f(p),g(p)) = \begin{cases} (2,1,1) & p=3 \\ (1,1,2) & p\equiv 1 \mod{3} \\ (1,2,1) & p \equiv 2 \mod{3} \end{cases}.
\end{align}
Therefore, it remains to determine the possible values of $(e_D(p),f_D(p),g_D(p))$.

Since $[K:\Q]=2$ is coprime to $[K_D:\Q]=3$ and $K'_D$ is the compositum of $K$ and $K_D$, we get that if $\mathfrak{P}$ is the prime dividing $p$ in $K$ that was fixed in Remark \ref{3splitrem}, then
$$(e_D(p),f_D(p),g_D(p)) = (e_{K'_D/K}(\mathfrak{P}),f_{K'_D/K}(\mathfrak{P}),g_{K'_D/K}(\mathfrak{P})).$$
A prime $\mathfrak{P}$ in $K$ ramifies in $K'_D$ if $\mathfrak{P}|D_1D_2^2$, splits if $D_1D_2^2$ is a cube modulo $\mathfrak{P}$ and is inert otherwise. Therefore,
\begin{align}\label{efgD1}
(e_D(p),f_D(p),g_D(p)) = \begin{cases} (3,1,1) & p|D \\ (1,1,3) & \left(\frac{D_1D_2^2}{\mathfrak{P}}\right)_3=1 \\ (1,3,1) & \left(\frac{D_1D_2^2}{\mathfrak{P}}\right)_3\not=0,1 \end{cases},
\end{align}
where $\left(\frac{\cdot}{\cdot}\right)_3$ is the cubic residue symbol for $K$.

Since $D_2=\sigma(D_1)$, where $\sigma$ is the generator of $\Gal(K/\Q)$, we get that
$$\left(\frac{D_2}{\mathfrak{P}}\right)_3 = \sigma\left(\frac{D_1}{\mathfrak{P}}\right)_3 = \left(\frac{D_1}{\mathfrak{P}}\right)_3^2.$$
Hence,
$$\left(\frac{D_1D_2^2}{\mathfrak{P}}\right)_3 = \left(\frac{D_1}{\mathfrak{P}}\right)_3^2.$$

Now, every integer can be written as $DD'$ where $D$ is $3$-split and all of the primes dividing $D'$ are $2 \mod{3}$. Define a multiplicative character on the integers as
\begin{align}\label{multchar}
\chi_p(DD') = \left(\frac{D_1}{\mathfrak{P}}\right)_3
\end{align}
then we can rewrite \eqref{efgD1} as
\begin{align}\label{efgD2}
(e_D(p),f_D(p),g_D(p)) = \begin{cases} (3,1,1) & p|D \\ (1,1,3) & \chi_p(D)=1 \\ (1,3,1) & \chi_p(D)\not=0,1 \end{cases}.
\end{align}
Note, that $\chi_p$ is \textit{not} a Dirichlet character.

\begin{rem}

In the case of $p=3$, everything will be a cube modulo $\mathfrak{P}_3$. Hence we have $\chi_3(D) = 1$ unless $3|D$ and therefore
$$(e_D(3),f_D(3),g_D(3)) = \begin{cases} (3,1,1) & 3|D \\ (1,1,3) & \mbox{otherwise} \end{cases}.$$
Further if $n$ is an integer such that all it's prime factors are $2 \mod{3}$ then $\chi_p(n)=1$.
\end{rem}

Putting everything together, we can write the $L$-functions of $K$ and $K_D$ as
\begin{align}\label{L_K2}
L_K(s) = \prod_{p\equiv 1 \mod{3}}\left(1-\frac{1}{p^s}\right)^{-1}\prod_{p\equiv 2 \mod{3}}\left(1+\frac{1}{p^s}\right)^{-1}
\end{align}
\begin{align}\label{L_D2}
L_D(s) = \prod_{\substack{p \\ \chi_p(D)=1}}\left(1-\frac{1}{p^s}\right)^{-2}\prod_{\substack{p \\ \chi_p(D)\not=0,1}}\left(1+\frac{1}{p^s}+\frac{1}{p^{2s}}\right)^{-1}
\end{align}

If $\chi$ is any character on $K$ modulo $\mathfrak{f}$, we define the $L$-function associated to this character as
$$L_K(\chi,s) = \prod_{\mathfrak{P}} \left(1-\frac{\chi_p(\mathfrak{P})}{N\mathfrak{P}^s}\right)^{-1}.$$
Finally, we will need a zero density theorem. We use Theorem 2.3 of \cite{K}.

\begin{thm}\label{Kovalthm}

For any $1/2\leq\alpha\leq1$ and $T>0$, let $N(\alpha,T,\chi)$ be the number of zeros $\rho=\beta+i\gamma$ of $L_K(\chi,s)$ with $\alpha\leq\beta\leq1$ and $|\gamma|\leq T$. Then there exists an $A>0$ such that
$$\sum_{q \leq Q} \mbox{  } \sideset{}{^*}\sum_{\chi \textnormal{ mod } q} N(\alpha,T,\chi) \ll (Q^2T)^{\frac{4(1-\alpha)}{3-2\alpha}} (\log QT)^A$$
where $\sideset{}{^*}\sum$ indicates that we sum over principal characters.
\end{thm}

\subsection{Explicit Formula}

Since $K_D$ has one embedding into $\R$ and two embeddings into $\C$, the function
$$\Lambda_D(s) := |\Delta_D|^{s/2}\Gamma_{\R}(s)\Gamma_{\C}(s)\zeta_D(s)$$
satisfies the functional equation
$$\Lambda_D(s) = \Lambda_D(1-s)$$
where
$$\Gamma_{\R}(s) = \pi^{-s/2}\Gamma(s/2) \quad \quad \quad \Gamma_{\C}(s) = 2(2\pi)^{-s}\Gamma(s) = \Gamma_{\R}(s)\Gamma_{\R}(s+1)$$
and $\Gamma(s)$ is the usual Gamma function.

Let $\rho_{D,j} = 1/2+i\gamma_{D,j}$ be the zeros of $L_D(s)$ and $f$ be an even Schwartz function. Proposition 2.1 of \cite{RS} gives the explicit formula
\begin{align}\label{explform}
\sum f(\gamma_{D,j}) = \frac{1}{2\pi} \int_{-\infty}^{\infty} f(x)\log \Delta_D dx- \frac{2}{2\pi} \sum_{n=1}^{\infty} \frac{\Lambda(n)\lambda_D(n)}{\sqrt{n}} \hat{f}\left(\frac{\log n}{2\pi}\right) + C_f
\end{align}
where the sum runs over all zeros of $L_D(s)$, $\Lambda(n)$ is the von-Magoldt function, $\lambda_D(n)$ satisfies
\begin{align}\label{lambda}
\frac{L'_D(s)}{L_D(s)} = -\sum_{n=1}^{\infty} \frac{\Lambda(n)\lambda_D(n)}{n^s}
\end{align}
and
\begin{align}
C_f = \frac{1}{2\pi} \int_{-\infty}^{\infty} f(x)\left(2\frac{\Gamma'_{\R}}{\Gamma_{\R}} \left(\frac{1}{2}+ix\right) + 2\frac{\Gamma'_{\R}}{\Gamma_{\R}}\left(\frac{1}{2}-ix\right) + \frac{\Gamma'_{\R}}{\Gamma_{\R}}\left(\frac{3}{2}+ix\right) + \frac{\Gamma'_{\R}}{\Gamma_{\R}} \left(\frac{3}{2}-ix\right)\right)dx
\end{align}
is independent of our choice of $D$.

Recalling the definition of $\mathscr{D}(K,f)$ from \eqref{numfieldOLD} requires multiplying the zeros by a factor of $L:=\frac{\log\Delta_D}{2\pi}$, we apply the explicit formula and the definition of $\Lambda(n)$ to get
\begin{align}\label{explform2}
\mathscr{D}(K_D,f)& = \sum f(L\gamma_{D,j}) \\
& = \int_{-\infty}^{\infty} f(x) dx- \frac{2}{\log\Delta_D} \sum_{m=1}^{\infty} \sum_p \frac{\lambda_D(p^m)\log p}{\sqrt{p^m}} \hat{f}\left(\frac{\log p^m}{\log\Delta_D}\right) + \widetilde{C_f}(D), \nonumber
\end{align}
where we use the observation that $\widehat{f(Lx)} = 1/L \hat{f}(x/L)$ and $\widetilde{C_f}(D)$ is the same as $C_f$ with $f$ replaced with $f(L\cdot)$.

\subsection{Main Term}

Applying the explicit formula \eqref{explform2}, we get
\begin{align}\label{mainterm}
\frac{1}{|\F_3(X)|}\sum_{K_D\in\F_3(X)} \mathscr{D}(K_D,f) & = \int_{-\infty}^{\infty} f(t)W(U)(t) dt-  ET
\end{align}
where
\begin{align}\label{ET1}
ET = \frac{1}{|\F_3(X)|}\sum_{K_D\in\F_3(X)} \left(\frac{2}{\log\Delta_D} \sum_{m=1}^{\infty} \sum_p \frac{\lambda_D(p^m)\log p}{\sqrt{p^m}} \hat{f}\left(\frac{\log p^m}{\log\Delta_D}\right) + \widetilde{C_f}(D)\right).
\end{align}
So it remains to show that $ET=O\left(\frac{1}{\log X}\right)$.

\section{Error Term}\label{appexplfor}

First, we note that if $K_D\in\F_3(X)$, then $X\leq \Delta_D\leq 2X$ and so $\log\Delta_D \sim \log X$ and we can rewrite \eqref{ET1}
\begin{align}\label{ET2}
ET \sim \frac{1}{c\sqrt{X}}\sum_{K_D\in\F_3(X)}\frac{2}{\log X}\left(\sum_{m=1}^{\infty} \sum_p \frac{\lambda_D(p^m)\log p}{\sqrt{p^m}} \hat{f}\left(\frac{\log p^m}{\log X}\right) + \widetilde{C_f}(D)\right)
\end{align}
where we also use \eqref{F3size} to write $|\F_3(X)|\sim c\sqrt{X}$.

\subsection{Easy Error Terms}

In this section, we show that most of terms of $ET$ are trivially $O\left(\frac{1}{\log X}\right)$.

By a change of variable in the definition $C_f$ we see that $\widetilde{C_f}(D) = O\left(\frac{1}{\log\Delta_D}\right)$ and hence
$$\frac{1}{c\sqrt{X}}\sum_{K_D\in\F_3(X)} \widetilde{C_f}(D) = O\left( \frac{1}{\sqrt{X}}\sum_{K_D\in\F_3(X)}\frac{1}{\log \Delta_D}\right) = O\left(\frac{1}{\log X}\right).$$

Now, we use the known bound $\lambda_D(p^m)=O(m)$ and the trivial bound $\hat{f}(x)=O(1)$ to get
\begin{align}\label{m>3}
& \frac{1}{c\sqrt{X}}\sum_{K_D\in\F_3(X)} \frac{2}{\log X}\sum_{m=3}^{\infty} \sum_p \frac{\lambda_{D}(p^m)\log p}{\sqrt{p^m}} \hat{f}\left(\frac{\log p^m}{\log X}\right) \nonumber \\
 \ll & \frac{1}{\sqrt{X}\log X}\sum_{K_D\in\F_3(X)}  \sum_p  \sum_{m=3}^{\infty} \frac{m\log p}{\sqrt{p^m}} \\
\ll & \frac{1}{\log X}\sum_p \frac{1}{p^{3/2-\epsilon}} = O\left(\frac{1}{\log X}\right) \nonumber
\end{align}

It remains to determine what happens for the sums when $m=1$ or $2$.

\subsection{Coefficients of $\frac{L_D'}{L_D}$}

Direct computation from \eqref{L_D2} shows
$$\lambda_D(p) = \lambda_D(p^2) = \begin{cases} 0 & \chi_p(D)=0 \\ 2 & \chi_p(D) =1 \\ -1 & \chi_p(D)\not=0,1  \end{cases}$$
Moreover, if $\chi_p$ is as in \eqref{multchar}, it is easy to see that
\begin{align}\label{lambda}
\lambda_D(p) = \lambda_D(p^2) = \chi_p(D) + \chi^2_p(D)
\end{align}
since $\chi_p$ is a cubic character.

Therefore we need to determine
\begin{align}\label{ET3}
\frac{1}{\sqrt{X}\log X}\sum_{K_D\in\F_3(X)} \sum_p \frac{\log p(\chi_p(D) + \chi^2_p(D))}{\sqrt{p^m}} \hat{f}\left(\frac{\log p^m}{\log X}\right)
\end{align}
for $m=1,2$.

Since $\chi_p$ is a cubic character, we have $\chi_p^2 = \overline{\chi_p}$. Hence it will be enough to determine
\begin{align}\label{ET4}
\frac{1}{\sqrt{X}\log X}\sum_{K_D\in\F_3(X)} \sum_p \frac{\chi_p(D)\log p}{\sqrt{p^m}} \hat{f}\left(\frac{\log p^m}{\log X}\right)
\end{align}
for $m=1,2$.

Applying Proposition \ref{classifyprop} we can write \eqref{ET4} as
\begin{align}\label{ET5class}
& \frac{1}{\sqrt{X}\log X}  \sideset{}{'}\sum_{\sqrt{X}\leq d_1d_2 \leq \sqrt{2X}} \sum_p \frac{\chi_p(d_1d_2^2)\log p}{\sqrt{p^m}} \hat{f}\left(\frac{\log p^m}{\log X}\right) \nonumber \\
+ & \frac{1}{\sqrt{X}\log X} \sideset{}{'}\sum_{\sqrt{X/81}\leq d_1d_2 \leq \sqrt{2X/81}} \sum_p \frac{\chi_p(3d_1d_2^2)\log p}{\sqrt{p^m}} \hat{f}\left(\frac{\log p^m}{\log X}\right) \\
+ & \frac{1}{\sqrt{X}\log X} \sideset{}{'}\sum_{\sqrt{X/81}\leq d_1d_2 \leq \sqrt{2X/81}} \sum_p \frac{\chi_p(9d_1d_2^2)\log p}{\sqrt{p^m}} \hat{f}\left(\frac{\log p^m}{\log X}\right) \nonumber
\end{align}
where $\sideset{}{'}\sum$ means we are summing over all pairs $d_1,d_2$ that are square-free, $3$-split, coprime and coprime to $3$. We see then it will be sufficient to determine
\begin{align}\label{ET6}
\frac{1}{\sqrt{X}\log X}  \sum_p \frac{\log p}{\sqrt{p^m}} \hat{f}\left(\frac{\log p^m}{\log X}\right)  \sideset{}{'}\sum_{d_1d_2 \leq Y} \chi_p(d_1d_2^2)
\end{align}
for $m=1,2$.

\subsection{Generating Series}

Fix a prime $p$ and consider the generating series
$$\mathcal{G}_p(s) = \sideset{}{'}\sum_{d_1,d_2} \frac{\chi_p(d_1d_2^2)}{(d_1d_2)^s}$$
which converges for $\Re(s)>1$.

It is tempting to treat $\mathcal{G}_p(s)$ as a multi-Dirichlet $L$-function. However, $\chi_p$ is \textit{not} a Dirichlet character. It is, however, related to a cubic Dirichlet character on $K=\Q(\zeta_3)$ modulo $\mathfrak{P}$. The following proposition shows exactly how $\mathcal{G}_p(s)$ is related to $L$-functions over $K$.

\begin{prop}\label{genserprop}

Let $\mathfrak{P}$ be the prime in $K$ dividing $p$ fixed in Remark \ref{3splitrem} and $\chi_{\mathfrak{P}} = \left(\frac{\cdot}{\mathfrak{P}}\right)_3$ be the cubic residue symbol modulo $\mathfrak{P}$ on $K$. Then
\begin{align}\label{analcont}
\mathcal{G}_p(s) = \sqrt{L_K(\chi_{\mathfrak{P}},s)L_K(\chi^2_{\mathfrak{P}},s)H_p(s)}
\end{align}
where $H_p(s)$ is some function (defined in the proof) that absolutely converges for $\Re(s)>1/2$.

\end{prop}

\begin{proof}

We can write an Euler product expansion for $\mathcal{G}_p(s)$ as follows
$$\mathcal{G}_p(s) =  \prod_{\ell\equiv 1 \mod{3}} \left(1 + \frac{\chi_p(\ell) + \chi^2_p(\ell)}{\ell^s} \right).$$

If $\mathfrak{l}$ is the fixed prime dividing $\ell$ in $K$, then we get $\chi_p(\ell) = \chi_{\mathfrak{P}}(\mathfrak{l})$. Moreover, we see that
$$\chi_p(\ell) + \chi^2_p(\ell) = \chi_{\mathfrak{P}}(\mathfrak{l}) + \chi^2_{\mathfrak{P}}(\mathfrak{l}) = \chi^2_{\mathfrak{P}}(\mathfrak{l}^{\sigma})+\chi_{\mathfrak{P}}(\mathfrak{l}^{\sigma}),$$
where $\sigma$ is the generator of $\Gal(K/Q)$. That is, the argument in the Euler product is independent of the choice of prime dividing $\ell$.

Further, if $\ell\equiv 1 \mod{3}$, then there always exists $2$ primes lying above it with $N\mathfrak{l}=\ell$. Thus
$$\prod_{\ell\equiv 1 \mod{3}} \left(1 + \frac{\chi_p(\ell) + \chi^2_p(\ell)}{\ell^s} \right) = \prod_{\substack{\mathfrak{l}|\ell \\ \ell \equiv1 \mod{3}}} \left(1 + \frac{\chi_{\mathfrak{P}}(\mathfrak{l}) + \chi^2_{\mathfrak{P}}(\mathfrak{l})}{N\mathfrak{l}^s}\right)^{1/2}.$$

Finally, if $\ell\equiv 2 \mod{3}$ then there exists a unique $\mathfrak{l}|\ell$ and $N\mathfrak{l}=\ell^2$. Therefore,
\begin{align*}
\prod_{\substack{\mathfrak{l}|\ell \\ \ell \equiv1 \mod{3}}} \left(1 + \frac{\chi_{\mathfrak{P}}(\mathfrak{l}) + \chi^2_{\mathfrak{P}}(\mathfrak{l})}{N\mathfrak{l}^s}\right) & = \prod_{\mathfrak{l}\not=\mathfrak{P}_3} \left(1 + \frac{\chi_{\mathfrak{P}}(\mathfrak{l}) + \chi^2_{\mathfrak{P}}(\mathfrak{l})}{N\mathfrak{l}^s}\right) \prod_{\ell\equiv 2\mod{3} } \left(1 + \frac{\chi_{\mathfrak{P}}(\mathfrak{l}) + \chi^2_{\mathfrak{P}}(\mathfrak{l})}{\ell^{2s}}\right)^{-1}\\
& = \prod_{\mathfrak{l}}\left(1-\frac{\chi_{\mathfrak{P}}(\mathfrak{l})}{N\mathfrak{l}^s}\right)^{-1}\prod_{\mathfrak{l}}\left(1-\frac{\chi^2_{\mathfrak{P}}(\mathfrak{l})}{N\mathfrak{l}^s}\right)^{-1} H_p(s) \\
& = L_K(\chi_{\mathfrak{P}},s)L_K(\chi^2_{\mathfrak{P}},s) H_p(s),
\end{align*}
where $H_p(s)$ is some Euler product that converges for $\Re(s)>1/2$.

\end{proof}

\begin{cor}\label{gensercor}

$$\sideset{}{'}\sum_{d_1d_2\leq Y} \chi_p(d_1d_2^2) = \int_{1-i\infty}^{1+i\infty} \mathcal{G}_p(s)\frac{Y^s}{s} ds$$

\end{cor}

\begin{proof}

We know that $L_K(\chi_{\mathfrak{P}},s)$ and $L_K(\chi^2_{\mathfrak{P}},s)$ are entire and zero free on $\Re(s)=1$. And since $H_p(s)$ can be written as an Euler product that converges for $\Re(s)>1/2$, it will also be analytic and zero free on $\Re(s)=1$. Hence $\mathcal{G}_p(s)$ will be analytic on $\Re(s)=1$. The result then follows from Perron's formula.

\end{proof}

The goal now is to analytically continue $\mathcal{G}_p(s)$ to a region to the left of $\Re(s)=1$ and move this contour integral as far as we can. Since we don't know anything about the convergence of $H_p(s)$ to the left of $\Re(s)=1/2$, the best we can hope to move the contour is to the line $\Re(s)=1/2+\epsilon$. Moreover, if $L_K(\chi_{\mathfrak{P}},s)$ has a zero, then the right hand side of \eqref{analcont} fails to be analytic at this zero.

Our plan moving forward is to move the contour for as many primes as we can and use Theorem \ref{Kovalthm} to bound the number of bad primes for which we can't move the contour. Of course GRH implies that we can move all the contours to the line $\Re(s)=1/2+\epsilon$ but we will refrain from using that for now.

\subsection{Bounding the Error Term}

\begin{prop}\label{calc}

Suppose $\supp(\hat{f})\subset (-\beta,\beta)$, then for any $T$ and $13/14 < \alpha <1$ we have
\begin{align*}
&\frac{1}{\sqrt{X}\log X}  \sum_p \frac{\log p}{\sqrt{p^m}} \hat{f}\left(\frac{\log p^m}{\log X}\right)  \sideset{}{'}\sum_{d_1d_2 \leq Y} \chi_p(d_1d_2^2) \\
\ll & \frac{X^{(\beta-1)/2+\epsilon}}{\log X}\left(\frac{Y}{T} + Y^{\alpha+\epsilon}\right) + \frac{Y(X^{2\beta}T)^{\frac{4(1-\alpha)}{3-2\alpha}}\left(\log XT\right)^A}{X^{(\beta+1)/2}}.
\end{align*}

\end{prop}

\begin{proof}

First of all, if $\supp(\hat{f})\subset (-\beta,\beta)$, then this will restrict the sum over the primes to the region $X^{\beta/m}$. Combining this with Corollary \ref{gensercor}, we get
\begin{align*}
& \frac{1}{\sqrt{X}\log X}  \sum_p \frac{\log p}{\sqrt{p^m}} \hat{f}\left(\frac{\log p^m}{\log X}\right)  \sideset{}{'}\sum_{d_1d_2 \leq Y} \chi_p(d_1d_2^2)\\
= & \frac{1}{\sqrt{X}\log X}  \sum_{p\leq X^{\beta/m}} \frac{\log p}{\sqrt{p^m}} \hat{f}\left(\frac{\log p^m}{\log X}\right)  \int_{1-i\infty}^{1+i\infty} \mathcal{G}_p(s)\frac{Y^s}{s} ds
 \end{align*}
We can write
$$ \int_{1-i\infty}^{1+i\infty} \mathcal{G}_p(s)\frac{Y^s}{s} ds=  \int_{1-iT}^{1+iT} \mathcal{G}_p(s)\frac{Y^s}{s}ds +  \int_{\substack{\Re(s)=1 \\ |\Im(s)|>T}} \mathcal{G}_p(s)\frac{Y^s}{s} ds.$$
Let $S_1$ be the sum consisting of the former and $S_2$ the latter. Then
\begin{align*}
S_2 & = \frac{1}{\sqrt{X}\log X}  \sum_{p\leq X^{\beta/m}} \frac{\log p}{\sqrt{p^m}} \hat{f}\left(\frac{\log p^m}{\log X}\right)  \int_{\substack{\Re(s)=1 \\ |\Im(s)|>T}} \mathcal{G}_p(s)\frac{Y^s}{s}\\
& \ll \frac{Y}{T\sqrt{X}\log X}  \sum_{p\leq X^{\beta/m}} \frac{\log p}{\sqrt{p^m}}\\
& \ll \frac{Y}{T\sqrt{X}\log X} \begin{cases} X^{\beta/2+\epsilon} & m=1 \\ \log X^{\beta/2} & m=2 \end{cases} \ll  \frac{YX^{(\beta-1)/2+\epsilon}}{T\log X}
\end{align*}

Define
$$\mathscr{E}_{\alpha}(Q,T) = \{p \leq Q : L_K(\chi_{\mathfrak{P}},s) \mbox{ has a zero in the region } \alpha<\Re(s)<1, |\Im(t)|<T\}.$$
Then we will write $S_1=S_3+S_4$, where $S_3$ consists of the sum of primes not in $\mathscr{E}_{\alpha}(Q,T)$ and $S_4$ consists of the sum of primes in $\mathscr{E}_{\alpha}(Q,T)$.

By definition, $\mathcal{G}_p(s)$ is analytic in the region $\alpha<\Re(s)<1, |\Im(t)|<T$ for $p\not\in\mathscr{E}_{\alpha}(X^{\beta},T)$, so we may shift the contour for these primes. That is
\begin{align*}
S_3 & = \frac{1}{\sqrt{X}\log X}  \sum_{\substack{p\leq X^{\beta/m} \\ p\not\in\mathscr{E}_{\alpha}(X^{\beta},T)}} \frac{\log p}{\sqrt{p^m}} \hat{f}\left(\frac{\log p^m}{\log X}\right)  \int_{1-iT}^{1+iT} \mathcal{G}_p(s)\frac{Y^s}{s}ds \\
& = \frac{1}{\sqrt{X}\log X}  \sum_{\substack{p\leq X^{\beta/m} \\ p\not\in\mathscr{E}_{\alpha}(X^{\beta},T)}} \frac{\log p}{\sqrt{p^m}} \hat{f}\left(\frac{\log p^m}{\log X}\right)  \left(\int_{\alpha+\epsilon-iT}^{\alpha+\epsilon+iT} \mathcal{G}_p(s)\frac{Y^s}{s}ds + \int_{\substack{\alpha+\epsilon\leq\Re(s)\leq1 \\ |\Im(s)|=T}}  \mathcal{G}_p(s)\frac{Y^s}{s} ds\right) \\
& \ll \frac{1}{\sqrt{X}\log X}  \sum_{\substack{p\leq X^{\beta/m} \\ p\not\in\mathscr{E}_{\alpha}(X^{\beta},T)}} \frac{\log p}{\sqrt{p^m}} \left(Y^{\alpha+\epsilon} + \frac{Y}{T}\right) \\
& \ll  \frac{1}{\sqrt{X}\log X} \left(Y^{\alpha+\epsilon} + \frac{Y}{T}\right)  \begin{cases} X^{\beta/2+\epsilon} & m=1 \\ \log X^{\beta/2} & m=2 \end{cases} \ll \frac{X^{(\beta-1)/2+\epsilon}}{\log X}\left(Y^{\alpha+\epsilon} + \frac{Y}{T}\right)
\end{align*}

Finally, recall that $N(\sigma,T,\chi)$ is the number of zeros of $L_K(\chi,s)$ in the region $\alpha<\Re(s)<1, |\Im(t)|<T$. Therefore, by Theorem \ref{Kovalthm}, we get for some $A>0$
$$|\mathscr{E}_{\alpha}(Q,T)| \leq \sum_{q \leq Q} \mbox{  } \sideset{}{^*}\sum_{\chi \textnormal{ mod } q} N(\sigma,T,\chi) \ll (Q^2T)^{\frac{4(1-\alpha)}{3-2\alpha}}\left(\log QT\right)^A.$$
Therefore,
\begin{align*}
S_4 &= \frac{1}{\sqrt{X}\log X}  \sum_{\substack{p\leq X^{\beta/m} \\ p\in\mathscr{E}_{\alpha}(X^{\beta/m},T)}} \frac{\log p}{\sqrt{p^m}} \hat{f}\left(\frac{\log p^m}{\log X}\right)  \int_{1-iT}^{1+iT} \mathcal{G}_p(s)\frac{Y^s}{s} ds \\
& \ll \frac{Y}{\sqrt{X}} \sum_{\substack{p \leq X^{\beta/m} \\ p\in\mathscr{E}_{\alpha}(X^{\beta/m},T) }} \frac{1}{\sqrt{p^m}}
\end{align*}

For $m=2$, we can bound the remaining sum by $\log X$ and get that $S_4 \ll \frac{Y}{\sqrt{X}}\log X$ which suffices. In order to manage when $m=1$, we will split it up into diadic intervals. Therefore
\begin{align*}
\sum_{\substack{X^{\beta}/2^j < p \leq X^{\beta}/2^{j-1} \\ p\in\mathscr{E}_{\alpha}(X^{\beta},T )}} \frac{1}{\sqrt{p}} & = \sum_{\substack{X^{\beta}/2^j < p \leq X^{\beta}/2^{j-1} \\ p\in\mathscr{E}_{\alpha}(X^{\beta}/2^{j-1},T )}} \frac{1}{\sqrt{p}} \\
& \ll |\mathscr{E}_{\alpha}(X^{\beta}/2^{j-1},T )|\frac{2^{j/2}}{X^{\beta/2}} \\
& \ll \frac{\left(X^{2\beta}T\right)^{\frac{4(1-\alpha)}{3-2\alpha}}\left(\log XT\right)^A}{X^{\beta/2}}2^{j/2\left(1 - \frac{16(1-\alpha)}{3-2\alpha}\right)}
\end{align*}

And so,
\begin{align*}
 S_4 & \ll \frac{Y\left(X^{2\beta}T\right)^{\frac{4(1-\alpha)}{3-2\alpha}}\left(\log XT\right)^A}{X^{(\beta+1)/2}} \sum_{j=1}^{\beta\log_2 X}2^{j/2\left(1 - \frac{16(1-\alpha)}{3-2\alpha}\right)} \\
& \ll \frac{Y\left(X^{2\beta}T\right)^{\frac{4(1-\alpha)}{3-2\alpha}}\left(\log XT\right)^A}{X^{(\beta+1)/2}}.
\end{align*}
This last line is true because the sum converges since $\alpha>13/14$ (and hence the exponent appearing is negative).

\end{proof}

\begin{cor}\label{calcor}

Assuming GRH we have
$$\frac{1}{\sqrt{X}\log X}  \sum_p \frac{\log p}{\sqrt{p^m}} \hat{f}\left(\frac{\log p^m}{\log X}\right)  \sideset{}{'}\sum_{d_1d_2 \leq Y} \chi_p(d_1d_2^2) \ll \frac{Y^{1/2+\epsilon}X^{(\beta-1)/2+\epsilon}}{\log X}$$

\end{cor}

\begin{proof}

Using the notation of the proof of Proposition \ref{calc}, GRH implies that $\mathscr{E}_{1/2}(Q,T)=\emptyset$ for all choices of $Q$ and $T$. Therefore $S_4=0$ and we can take $T\to\infty$ to get that $S_2=0$ and
$$S_3 \ll \frac{Y^{1/2+\epsilon}X^{(\beta-1)/2+\epsilon}}{\log X}.$$

\end{proof}

\subsection{Proof of Theorem \ref{Mainthm}}

Now, we can finally prove Theorem \ref{Mainthm}.

\begin{proof}[Proof of Theorem \ref{Mainthm}]

By Proposition \eqref{calc} and \eqref{mainterm} we see that if $\supp(\hat{f})\subset (-\beta,\beta)$, then
$$\frac{1}{|\F_3(X)|}\sum_{K_D\in\F_3(X)} \mathscr{D}(K_D,f) = \int_{-\infty}^{\infty} f(t)W(U)(t)dt - ET$$
where for any $T>0$ and $13/14<\alpha<1$,
$$ET \ll \frac{X^{(\beta-1)/2+\epsilon}}{\log X}\left(\frac{X^{1/2}}{T} + X^{\alpha/2+\epsilon}\right) + \frac{X^{1/2}(X^{2\beta}T)^{\frac{4(1-\alpha)}{3-2\alpha}}\left(\log XT\right)^A}{X^{(\beta+1)/2}} + \frac{1}{\log X}.$$
Setting $T=X^{\beta}$, we get
$$ET \ll \frac{1}{X^{\beta/2-\epsilon}\log X} + \frac{X^{(\alpha+\beta-1)/2+\epsilon}}{\log X} + X^{\beta\left(\frac{12(1-\alpha)}{3-2\alpha}-\frac{1}{2}\right)}\left(\log X\right)^A + \frac{1}{\log X}.$$
Since $\alpha>13/14$, we get that $\frac{12(1-\alpha)}{3-2\alpha}-\frac{1}{2}<0$ and so the only restriction on $\beta$ comes from the second term. That is as long as $\beta<1-\alpha<1/14$ we have
$$ET \ll \frac{1}{\log X}.$$

If we assume GRH then by Corollary \ref{calcor} we get
$$ET \ll \frac{X^{(\beta-1/2)/2+\epsilon}}{\log X} + \frac{1}{\log X} $$
and as long as $\beta<1/2$, we get $ET \ll \frac{1}{\log X}$.

\end{proof}

\bibliography{CubicCasegood}
\bibliographystyle{amsplain}

\end{document}